\documentclass[11pt]{amsart}
\usepackage{amssymb, amsmath, amsthm}
\usepackage{amsrefs}
\usepackage[latin1]{inputenc}
\usepackage{graphicx}
\usepackage[final]{hyperref}
\usepackage[a4paper, centering]{geometry}
\usepackage{color}
\usepackage{graphicx} 

\newtheorem{theorem}{Theorem}
\newtheorem{proposition}[theorem]{Proposition}
\newtheorem{lemma}[theorem]{Lemma}
\newtheorem{corollary}[theorem]{Corollary}
\theoremstyle{definition}
\newtheorem{definition}[theorem]{Definition}
\theoremstyle{remark}

\def\R{\mathbb{R}}

\def\N{\mathbb{N}}

\def\pscal#1#2{\left\langle#1,\,#2\right\rangle}

\def\dist{d_{\partial \Omega}}

\def\d{\delta}
\def \e{\varepsilon}

\def\X{\mathbf{X}}
\def\Xe{\X_{\varepsilon}}

\DeclareMathOperator{\cut}{\overline \Sigma}  
\DeclareMathOperator{\high}{M}

\DeclareMathOperator{\argmax}{argmax}

\begin{document}

\title[Infinity ground states]%
{Rigidity results \\ for variational infinity ground states}%
\author[G.~Crasta, I.~Fragal\`a]{Graziano Crasta,  Ilaria Fragal\`a}
\address[Graziano Crasta]{Dipartimento di Matematica ``G.\ Castelnuovo'', Univ.\ di Roma I\\
P.le A.\ Moro 2 -- 00185 Roma (Italy)}
\email{crasta@mat.uniroma1.it}

\address[Ilaria Fragal\`a]{
Dipartimento di Matematica, Politecnico\\
Piazza Leonardo da Vinci, 32 --20133 Milano (Italy)
}
\email{ilaria.fragala@polimi.it}

\date{June 28, 2017}

\begin{abstract}  We prove two rigidity results for a variational infinity ground state $u$ of an open bounded convex domain $\Omega \subset \R ^n$. 
They state that $u$ coincides with a multiple of the distance from the boundary of $\Omega$ if either $|\nabla u|$ is constant on $\partial \Omega$, or $u$ is of class $C ^ {1,1}$ outside the high ridge of $\Omega$.  Consequently, in both cases $\Omega$ can be geometrically characterized as a ``stadium-like domain''. 
\end{abstract}

\maketitle

\section{Introduction}

Let $\Omega$ be an open bounded subset of $\R ^n$.  A function $u: \Omega \to \R$ is called a {\it variational infinity ground state} if there exists a sequence $p _j \to + \infty$ such that, denoting by $u _{p _j}$ the first Dirichlet eigenfunction of the $p_j$-Laplacian, there holds  
$$\lim _j u _{p _j} = u \qquad \text{ uniformly in } \overline \Omega\,.$$ 
Recall that, for any $p>1$, $u_p$ is given by the unique 
solution to the minimization problem 
\begin{equation}\label{p:lp}
\Lambda _ p  = \inf \Big \{ \frac{ \int _{\Omega} |\nabla u| ^ p}{ \int _{\Omega} | u| ^ p}  \ :\ u \in W ^ {1, p} _ 0 (\Omega) \, , { \int _{\Omega} | u| ^ p}  = 1 \Big \} \,. 
\end{equation}
By passing to the limit as $p \to + \infty$ in the Euler-Lagrange equation for problem \eqref{p:lp}, namely
$$- {\rm div} \big ( |\nabla u| ^ p \nabla u ) = \Lambda _p ^ p  |u| ^ { p-2} u \, , $$
it was proved by Juutinen, Lindqvist and Manfredi in \cite{JLM} that a variational infinity ground state is a viscosity solution to
\begin{equation}
\label{ground}
\min\left\{\frac{|Du|}{u} - \Lambda_\infty,\ -\Delta_\infty u 
\right\} = 0\, .
\end{equation}
Here $\Delta_\infty$ is the infinity Laplace operator, which is defined for smooth functions $u$ by
$$\Delta _\infty u := \nabla ^ 2 u \cdot \nabla u \cdot \nabla u\,$$ 
and
\[
\Lambda _ \infty := \lim _ p \Lambda _ p =  \big ( \max _{\overline \Omega} \dist (x) \big ) ^ {-1}\,,
\] 
where $d _{\partial \Omega}$ denotes the distance function from the boundary of $\Omega$. 

Since the pioneering paper \cite{JLM}, the infinity eigenvalue problem \eqref{ground} has been further studied in the literature \cite{CDJ, HSY, JLM2, NRSS, Yu},  and  a viscosity solution to it  is called an {\it infinity ground state}. 

We point out that, according to a counterexample given in \cite{HSY}, a non-convex domain may possess an infinity ground state which is not variational. On the other hand, on convex domains, 
the uniqueness  of solutions to problem \eqref{ground} up to some constant factor has not yet been proved or disproved (and the same holds for the uniqueness of variational infinity ground states). 
An exception is represented by the case when the distance function from the boundary is an infinity ground state: in this case, it turns out that constant multiples of $d _{\partial \Omega}$ are the only  infinity ground states (and actually the same assertion remains true on possibly non-convex domains provided the set ${\rm argmax}_{\overline \Omega} (d _{\partial \Omega})$ is connected). 

Such uniqueness result was proved in \cite{Yu} by Yu, who also observed that a necessary and sufficient condition on the geometry of $\Omega$ in order that $d _{\partial \Omega}$ is an infinity ground state is the coincidence between the high ridge $\high (\Omega)$ and the cut locus $\cut (\Omega)$ of $\Omega$ (see Section \ref{secprel} for the precise definitions).   Later on, this class of domains has been completely characterized in our paper \cite{CFb}, under the assumption that the space dimension is $n=2$
(or in higher dimensions if working within the restricted class of convex sets). In particular we proved that, for $n = 2$, cut locus and high ridge coincide if and only if the domain is the  tubular neighbourhood of a $C ^ {1,1}$ manifold (with or without boundary).  Since for convex domains this amounts to say that $\Omega$ is the parallel set of a line segment, possibly degenerated into a point, we call  
any domain with $\high(\Omega) = \cut(\Omega)$ a {\it stadium-like domain}. 

This paper is devoted to prove some rigidity results for variational infinity ground states on convex domains of $\R ^n$. 
Namely, we individuate  some sufficient conditions in order that a variational infinity ground state on a convex set $\Omega\subset \R ^n$ coincides with the distance function from its boundary, and consequently that $\Omega$ is a stadium-like domain and there are no further infinity ground states.

The first sufficient condition is an overdetermined boundary datum for the gradient of $u$ along the boundary, yielding
the following  Serrin-type theorem:

\begin{theorem}\label{thm1}
Let $\Omega$ be an open bounded convex subset of $\R ^n$, and let $u$ be a variational infinity ground state in $\Omega$. 
Assume that $u$ is of class $C ^ 1$ on $\{ x \in \overline \Omega \, :\, \dist (x) < \d\}$  for some $\d >0$. 
If there exists a positive constant $c$ such that
\[
|\nabla  u| = c \hbox{ on } \partial \Omega\, ,  
\]
 then $c = \Lambda_\infty$, $u = c \, \dist$, and $\Omega$ is a stadium-like domain.  
\end{theorem}

The second condition is of completely different kind, as it involves the regularity of $u$ outside the high ridge. 
Let us stress that, to the best of our knowledge,  no regularity result is available for infinity ground states. The 
only remark in this direction appears in Section~4 of \cite{JLM2}, where the authors proved that,  in a two-dimensional square, there are no
variational infinity ground states 
which are of class $C^2$ in the punctured square. This is clearly encompassed by the following result:

\begin{theorem}\label{thm2}
	Let $\Omega$ be an open bounded convex subset of $\R ^n$, and let $u$ be a variational infinity ground state in $\Omega$.
	Assume that $u\in C^{1,1}(\Omega\setminus M(\Omega))$, where $M (\Omega) := \argmax_{\overline \Omega} (\dist)$.
	Then $ u$ is a multiple of  $\dist$, and $\Omega$ is a stadium-like domain. 
\end{theorem}

By combining Theorems \ref{thm1} and \ref{thm2} with the results proved in \cite{CFb}, we obtain:

\begin{corollary}\label{cor12}
Let $\Omega$ be an open bounded convex subset of $\R ^n$, and let $u$ be a variational infinity ground state in $\Omega$
which satisfies the assumptions of Theorem \ref{thm1} or of Theorem \ref{thm2}. 
Then:

\smallskip
\begin{itemize}
\item[(i)] if $n=2$, $\Omega$ is the parallel neighbourhood of a line segment, possibly reduced to one point; 

\smallskip
\item[(ii)] if $n \geq 2$ and $\partial \Omega$ is of class $C ^2$, then $\Omega$ is a ball.
\end{itemize}

\end{corollary}

Our proofs of Theorems \ref{thm1} and \ref{thm2} are based on a gradient flow technique, which in case of Theorem \ref{thm2} is applied directly to $u$, while in case of Theorem \ref{thm1} is applied to its regularizations via supremal convolution. 
This approach was firstly introduced in \cite{CFd} in order to study inhomogeneous 
overdetermined boundary value problems for the infinity Laplacian. 
Here the additional difficulty is the dichotomy appearing in equation \eqref{ground}, 
or in other terms the fact that we have to deal with singular points of infinity ground states. 
A further obstacle is represented by the fact that the operator $F(u,p,X) := \min\{ \frac{|p|}{  u} - \Lambda_\infty, \ -\pscal{Xp}{p}\}$
is decreasing in the $u$ variable (for $u > 0$); for this reason, 
the typical result ensuring that the supremal convolutions are viscosity supersolutions (see {\it e.g.}\ \cite[Theorem 10]{Kat}) does not apply to  problem \eqref{ground}.

The key argument we enforce to overcome these difficulties is the crucial property of variational infinity ground states on convex domains of being log-concave, and consequently locally semiconcave. This enables us to work, rather than in the setting of viscosity solutions to \eqref{ground}, in the setting of 
solutions to the eikonal equation. The advantage is twofold:  it is equivalent to deal with
almost everywhere solutions or viscosity solutions, and we can invoke a comparison principle leading to rigidity.  

In fact, let us remark that the role of the convexity assumption on $\Omega$ in Theorems \ref{thm1} and \ref{thm2} is 
precisely to ensure that $u$ is log-concave  and hence locally semiconcave. 

Moreover, the reason why we cannot state our results for arbitrary infinity ground states (possibly not variational) on convex domains,
is the lack of information about their local semiconcavity.  
We address the extension of Theorems \ref{thm1} and \ref{thm2} to the not variational case as an interesting open problem. 

The contents are organized as follows. 
In order to make the paper self-contained, we start by recalling in Section \ref{secprel} the basic facts about variational infinity ground states which intervene in the proofs of Theorems \ref{thm1} and \ref{thm2}, along with the geometric results on the distance function which allow to deduce Corollary \ref{cor12}.  
Then the proofs of Theorems~\ref{thm1} and~\ref{thm2} are given respectively in Sections \ref{proof1} and \ref{proof2}.

\section{Background material}\label{secprel}

Hereafter  we recall the main results we shall need to exploit about variational infinity ground states. 
We state them as a series of propositions. 

\begin{proposition}\label{p:semiconc}
Let \(u\colon A\to\R\) be a locally Lipschitz function on an open set \(A\subset\R^n\),
and let \(\psi\in C^2(I)\) be a function satisfying \(\psi'(t) > 0\) for every \(t\in I\), 
where \(I\subset\R\) is an open interval containing \(u(A)\).
If the composite function \(v = \psi\circ u\) is concave on every convex subset of \(A\),
then \(u\) is locally semiconcave in \(A\).
\end{proposition}

\begin{proof}
We remark that \(v\) is a locally Lipschitz function, and
\(\nabla v = \psi'(u)\, \nabla u \in L^\infty_{\text{loc}}(A)\).
Let \(B\Subset A\) and denote \(K := \overline{B}\).
Since \(\psi'(u)\) is Lipschitz continuous in \(K\),
there exists a sequence \((\zeta_k)_k \subset C^\infty(A)\) such that
\[
\zeta_k \to \psi'(u)\quad\text{uniformly in}\ K,\qquad
\nabla\zeta_k \to \nabla(\psi'(u)) = \psi''(u) \nabla u
\quad \text{in}\ L^1(K).
\]
Let us define the constants
\[
L := \sup_K |\nabla u|,
\qquad
C := \min_{u(K)} \psi'',
\qquad
c := \min_{u(K)} \psi' > 0.
\]

Let us fix a unit vector \(\xi\in\R^n\).
Since, by assumption, \(v\) is concave in \(A\), one has
\(D^2 v \leq 0\) in distributional sense.
Since \(\nabla u \in L^\infty(K)\),
then for every \(\varphi\in C^\infty_c(B)\), \(\varphi\geq 0\), 
it holds:
\[
\begin{split}
0 & \geq \int_K D_{ij} v \, \xi_i \xi_j\, \varphi
= - \int_K \psi'(u) D_i u D_j\varphi\, \xi_i \xi_j
= - \lim_k \int_K \zeta_k D_i u D_j \varphi \, \xi_i \xi_j
\\
& = \lim_k \pscal{D_j(\zeta_k D_i u \, \xi_i \xi_j)}{\varphi}
= \lim_k \left[\pscal{D_{ij} u \, \xi_i \xi_j)}{\zeta_k\varphi}
+ \int_K D_j \zeta_k\, D_i u\, \varphi \, \xi_i \xi_j\right]
\\
& = \lim_k \pscal{D_{ij} u \, \xi_i \xi_j}{\zeta_k\varphi}
+ \int_K \psi''(u) (\nabla u \cdot \xi)^2 \, \varphi\,.
\end{split}
\]
From the above inequality we get that,
for every \(\e \in (0, c) \), there exists an index \(k_\e\in\N\) such that
\[
\pscal{D_{ij} u \, \xi_i \xi_j}{\zeta_k\varphi}
 \leq |C| \, L^2 \int_K \varphi + \e,
 \qquad \forall k > k_\e,\
 \forall \varphi\in C^\infty_c(B).
\]
Since \(\zeta_k\to \psi'(u) \geq c > 0\) uniformly in \(K\),
we can choose \(k_\e\) so that
\(\zeta_k \geq c - \e\) for every \(k > k_\e\), so that
\(\frac{1}{\zeta_k} \in C^\infty(K)\), hence we can use the test function
\({\varphi}_k := \varphi / \zeta_k\) in the above relation to get
\[
\pscal{D_{ij} u \, \xi_i \xi_j}{{\varphi}}
 \leq \frac{|C| \, L^2}{c-\e} \int_K \varphi + \e,
 \quad \forall \varphi\in C^\infty_c(B).
\]
By the arbitrariness of \(\e\) we conclude that \(D_{ij} u \leq \frac{|C|\, L^2}{c}\, I\) in
the sense of distributions in \(B\),
hence \(u\) is locally semiconcave.
\end{proof}

\begin{proposition}\label{p:Yu1}
Let $\Omega$ be an open bounded convex subset of $\R ^n$, and let $u$ be a variational infinity ground state in $\Omega$.
Then $\log u$ is  concave and consequently $u$ is locally semiconcave, namely for every compact subset $K$ of $\Omega$ there exists a positive constant $C= C (K)$ such that 
\[
u(\lambda x + (1-\lambda)y) \geq \lambda u(x) + (1-\lambda) u(y)
-C\frac{\lambda(1-\lambda)}{2}\, |x-y|^2
\qquad \forall [x,y] \subset K \ \text { and } \ \forall \lambda \in [0,1]\,.
\]
\end{proposition}
\proof By a well-known result due to Sakaguchi, for every $p >1$ the first Dirichlet eigenfunction $u _p$ of the $p$-Laplacian is log-concave \cite{Sak}.
Then the same assertion holds true by its definition for a variational infinity ground state. 
The local semiconcavity of \(u\) now follows from Proposition~\ref{p:semiconc}.
\qed

\bigskip
\begin{proposition}\label{p:Yu2}
Let $\Omega$ be an open bounded subset of $\R ^n$, and let $u$ be a variational infinity ground state in $\Omega$
normalized by $\max u = \max d_{\partial\Omega}$. 
Then:

\smallskip
\begin{itemize}

\item[(i)] for every $x \in \Omega$, there holds $u(x) \leq \dist(x)$; 

\smallskip
\item [(ii)] if $x \in {\rm argmax} _{\overline \Omega} ( d _{\partial \Omega})$, and $y \in \Pi (x):= \{ y \in \partial \Omega \ :\ d _{\partial \Omega } (x) = |x-y| \}$, 
then $u (z) = d _{\partial \Omega } (z)$ for every $z \in [x, y]$; 

\smallskip
\item[(iii)] it holds $ {\rm argmax} _{\overline \Omega} ( u) =  {\rm argmax} _{\overline \Omega} ( d _{\partial \Omega})$.

\end{itemize}
\end{proposition}

\proof For (i), see \cite[Section 1]{JLM2}. 
For (ii) and (iii), see \cite[Theorem 2.4]{Yu}.  \qed

\bigskip

\begin{proposition}\label{p:Yu3}
Let $\Omega$ be an open bounded subset of $\R ^n$, and let $u$ be a variational infinity ground state in $\Omega$.
Then the function 
\[
S ^- (x):=  \lim _{r \to 0 ^+} \max _{y \in \partial B _ r (x)} \frac{u (x) - u (y) }{r} 
\]
is continuous in \(\Omega\) and
agrees with $|\nabla u (x)|$ if $u$ is differentiable at $x$. 
\end{proposition}

\proof See \cite[Theorem 3.6]{Yu}.  \qed

\bigskip
We conclude with the precise definition of stadium-like domains mentioned in the Introduction, along with their characterization according to our previous paper \cite{CFb}.

\begin{definition} Let $\Omega$ be an open bounded domain in $\R ^n$. Let $\Sigma (\Omega)$ denote the set of points in $\Omega$ where $\dist$ is not differentiable. 
The {\it cut locus} $\cut (\Omega)$  is the closure of $\Sigma (\Omega)$ in $\overline \Omega$. 
The {\it high ridge} $\high (\Omega)$ is the set of points in $\overline \Omega$ where $\dist$ attains its maximum. 
We say that $\Omega$ is a {\it stadium-like domain} if $\cut (\Omega) = \high (\Omega)$. 
\end{definition}

\begin{proposition} 
	Let $\Omega$ be an open bounded set in $\R ^n$ and assume that it is a stadium-like domain. 

\begin{itemize}
\item[(i)] In dimension $n  =2$, $\Omega$ is either a disk
	or a parallel neighbourhood  of a $1$-dimensional $C ^ {1,1}$ manifold.
If in addition $\Omega$ is $C ^2$, then $\Omega$ is either a disk
		or a parallel neighborhood  of a $1$-dimensional $C ^ {2}$ manifold with no boundary; in particular, 
		if $\Omega$ is also simply connected, then $\Omega$ is a disk.

\smallskip
\item[(ii)]	In any dimension $n \geq 2$, if $\Omega$ is convex and of class  
	$C^2$, then 
	$\Omega$ is a ball.
	\end{itemize}

\end{proposition}
\proof See \cite[Theorem 6 and Theorem 12]{CFb}. \qed

\section{Proof of Theorem \ref{thm1}}\label{proof1}

Our approach is based on the use of the supremal convolutions of $u$, and more precisely on the study of the behaviour of their gradient, in modulus, along the gradient flow.  

Recall that the supremal convolutions of $u$ are defined for $\e >0$ by
\begin{equation}\label{f:ue}
u ^ \e (x) := \sup _{y \in \Omega} \Big \{ u (y)  - \frac { |x-y| ^ 2 }{2 \e}  \Big \} \qquad \forall x \in \Omega\,.
\end{equation}

Let us start with a preliminary lemma in which we recall
some basic well-known properties of the functions $u ^ \e$. 
To fix our setting let us recall that, 
setting
\[
\rho(\e) := 2 \sqrt{\e \, \|u\|_\infty},
\qquad
 \Omega^{\rho(\e)} := \{x\in\Omega:\ d_{\partial\Omega}(x) > \rho(\e)\}\,,
\]
then for every \(x \in \Omega^{\rho(\e)}\) the supremum in \eqref{f:ue}
is attained at a point \(y_\e(x) \in \overline{B}_{\rho(\e)}(x)\subset\Omega\).
Thus, setting
\begin{equation}\label{f:Ue} 
U _\e:= \big \{ x \in \Omega \ :\ u (x) > \e \big \} \, , \qquad 
A_\e := \big \{ x \in U _\e \ :\ d_{\partial U _\e}(x) > \rho(\e) \big \}\, ,
\end{equation}
there holds
\begin{equation}\label{d:ue2}
u ^ \e (x) = 
u(y_\e(x)) - \frac{|x-y_\e(x)|^2}{2\e} =
\sup _{y \in U _\e} \Big \{ u (y)  - \frac { |x-y| ^ 2 }{2 \e}  \Big \} 
\qquad \forall x \in \overline A_\e\,. 
\end{equation}
In what follows, we shall always assume that $\e \in (0, 1)$ is small enough
to have $A_\e \neq \emptyset$.
Moreover, let us define
\begin{equation}\label{f:omegae}
m_\e := \max_{\partial A_\e} u^\e,
\qquad
\Omega_\e := \{x\in A_\e : \ u^\e (x) > m_\e \}\,.
\end{equation}

\begin{lemma}\label{l:approx1} 
Let $\Omega$ be an open bounded convex subset of $\R ^n$, 
let $u$ be a variational infinity ground state in $\Omega$, and let $u ^ \e$ be the supremal convolutions defined in \eqref{f:ue}. Then:

\begin{itemize}
\item[(i)] for every $\e >0$, $u ^ \e$ is of class $C ^ {1, 1}$ in 
a open neighbourhood of $\overline{\Omega _\e}$;

\smallskip
\item[(ii)] 
$ u ^ \e \to u$ locally uniformly in ${\Omega}$ as $\e \to 0^+$
(so that $m_\e \to 0$ and $\Omega_\e$ converges to $\Omega$
in Hausdorff distance);

\smallskip
\item[(iii)] if $u$ is of class $C^1$ in a 
compact set \(K\subset\Omega\), 
then $\nabla u^\e \to \nabla u$ uniformly on \(K\)
as $\e \to 0^+$;

\smallskip
\item[(iv)]
if $u$ is differentiable at $x\in\Omega$, then $\lim_{\e \to 0^+} \nabla u^\e (x) = \nabla u (x)$.
 
\end{itemize}
\end{lemma}

\begin{proof} 
For (i) and (ii), we refer to \cite[Lemma 4]{CFd}, \cite[Thm. 3.5.8]{CaSi}. 

(iii) For every $x\in\Omega_\e$, 
since \(u^\e\) is differentiable at \(x\) we have that
the point \(y_\e(x)\) is characterized by
$y_\e(x) := x + \e \, \nabla u_\e(x)$.
Moreover, 
from the magic property of super-jets ({\it cf.} \cite[Lemma A.5]{CHL}),  
$\nabla u^\e(x)$ belongs to $D^+ u(y_\e(x))$, where $D ^+ u (y)$ denotes as usual the superdifferential of $u$ at $y$, namely
\[
D^+ u (y) := 
\left\{ p \in \R ^n \ :\ \limsup _{ z \to y} 
\frac{u(z) - u(y) -\langle p, z-y \rangle }{|z-y| }  
\leq 0  \right\} \,.
\]
Let $\e_0>0$ be such that
$K_0 := K + \overline{B}_{\rho(\e_0)} \subset \Omega$.
Since $\nabla u$ is uniformly continuous on $K_0$,
there exists a modulus of continuity $\omega\colon [0, +\infty) \to [0, +\infty)$
such that
\[
|\nabla u(x) - \nabla u(y)| \leq \omega(|x-y|),
\qquad \forall x,y \in K_0.
\]
Hence, if $\e \in (0, \e_0)$, we have that
\[
\max_{x\in K} |\nabla u^\e(x) - \nabla u(x)|
= \max_{x\in K} |\nabla u(y_\e(x)) - \nabla u(x)|
\leq \max_{x\in K} \omega(|y_\e(x) - x|)
\leq \omega(\rho(\e)).
\]

(iv) With the same notation of (iii) we have that, for $\e$ small enough,
\[
\nabla u^\e (x) =: p_\e \in D^+ u(y_\e(x)).
\]
By the upper semicontinuity of the super-differential
(see \cite[Prop.~3.3.4(a)]{CaSi}), 
since $y_\e(x) \to x$ we have that
any cluster point of $(p_\e)$ belongs to $D^+ u (x) = \{\nabla u(x)\}$,
hence $p_\e \to \nabla u(x)$.
\end{proof}

\bigskip

In the following we are going to use a gradient flow technique not far from the one we have already
successfully adopted in our previous papers
\cite{CFd,CFe,CFf}.
Let us consider the (normalized) gradient flow of $u ^ \e$, namely the family of curves $\gamma _\e$ which solve the Cauchy problems
\begin{equation}\label{f:geoe}
\begin{cases}
\dot \gamma_\e (t) = \dfrac{\nabla u ^ \e (\gamma _\e (t))}{|\nabla u ^ \e (\gamma _\e (t))|}
\,,
\\
\gamma_\e  (0) = x _\e \in \overline \Omega _\e\,. 
\end{cases}
\end{equation}
Clearly the solution to the Cauchy problem above is well defined if and only if
$\nabla u^\e (x_\e) \neq 0$.
We are going to see, in Lemma~\ref{l:prel} below,
that the set $\{x\in\Omega_\e:\ \nabla u^\e (x_\e) = 0\}$
is independent of $\e$ and coincides with
$M(\Omega) := \argmax_{\overline \Omega} (u)$.

Hence, for every $x_\e \in \overline \Omega _\e\setminus M(\Omega)$, problem \eqref{f:geoe} admits a unique maximal solution 
\[
\Xe(\cdot, x_\e)\colon [0, T^+_\e(x_\e)) \to \overline{\Omega_\e}\setminus M(\Omega), 
\]
which will be called a {\it trajectory} associated with $u^\e$.
Indeed, 
the function $t\mapsto u^\e(\gamma_\e(t))$ is
non-decreasing, since
\[
\frac{d}{dt}u^\e (\gamma_\e(t)) 
= |\nabla u^\e(\gamma_\e(t))| \geq 0,
\] 
so that 
the maximal existence time is characterized by
\[
\lim_{t \uparrow T^+_\e(x_\e)}
u^\e (\gamma_\e(t)) \in M(\Omega),
\]
while uniqueness follows from the
$C^{1,1}$ regularity of $u^\e$
stated in Lemma~\ref{l:approx1} (i).

In the following we shall heavily use some properties of the gradient flow
associated with an
$\infty$--subharmonic regular function.
These properties are proved by means of a discrete approximation of the flow
that we quickly recall below
(see e.g.\ \cite[Proposition~6.2]{Cran} for the details).

Assume that $u$ is an $\infty$--subharmonic function 
of class $C^{1,1}$
in an open set $Q\subset\R^N$.
Let us fix $\delta > 0$.
Given $x_0\in Q$, we can construct a sequence $x_0, x_1, \ldots, x_N$ of points in $Q$
in such a way that
\[
|x_{j} - x_{j-1}| = \delta,
\quad
u(x_{j}) = \max_{\partial B_\delta(x_{j-1})} u,
\qquad
\forall j = 1,\ldots, N.
\]
The terminal point $x_N$ is chosen so that $\overline{B}_\delta(x_N) \notin Q$
(otherwise the sequence is infinite).
Now, let $\gamma_\delta$ be the piecewise linear curve with unitary speed connecting
the points $x_0, x_1, \ldots, x_N$.
Since $u\in C^{1,1}(Q)$,
as $\delta\to 0$ the curve $\gamma_\delta$ converges to the unique solution
of the normalized gradient flow equation with initial data $\gamma(0) = x_0$,
and the following additional properties hold:
\begin{itemize}
	\item[(a)]
	$t \mapsto |\nabla u(\gamma(t))|$ is a non-decreasing function;
	\item[(b)]
	$t\mapsto u(\gamma(t))$ is a convex function.
\end{itemize}
Clearly, if $u$ is an $\infty$--harmonic function in $Q$, we can conclude that
$t\mapsto |\nabla u(\gamma(t))|$ is constant and
$t\mapsto u(\gamma(t))$ is an affine function.

\begin{lemma}\label{l:prel}
Let $\Omega$ be an open bounded convex subset of $\R ^n$, 
let $u$ be a variational infinity ground state in $\Omega$, 
and let $u ^ \e$ be the supremal convolutions defined by \eqref{f:ue}.  
For $\e>0$, let $\mu_\e\in (m_\e, \max u)$ and
let $M _\e = \{u^\e > \mu_\e\}$.
Then:

\begin{itemize}
\item[(i)] it holds 	
\begin{equation}\label{f:max}
M(\Omega):=\argmax_{\overline \Omega} (u) = \argmax_{\overline \Omega_\e} (u^\e)
\qquad\text{and}\qquad
\max_{\overline \Omega} (u) = \max_{\overline \Omega_\e} (u^\e);
\end{equation}

\smallskip

\item[(ii)] 
$M_\e$ is a neighbourhood of $M(\Omega)$, and
every trajectory 
starting from a point $x_\e \in \partial\Omega_\e$
enters $M _\e$ in a finite time $T _\e (x_\e)$
(that is, it holds $\Xe (t, x _\e ) \in \overline{\Omega}_\e \setminus M _\e$ for $t \in [0, T _\e (x_\e )]$ and  
$\Xe (t, x _\e ) \in  M _\e$ for  $T _\e (x_\e ) < t < T^+_\e(x_\e)$,
for every $x _\e \in \partial\Omega_\e$);

\item[(iii)]
the set $\Omega _\e \setminus M _\e$ is covered by trajectories starting at points of  $\partial \Omega _\e \setminus \Gamma _\e$.

\end{itemize}

\end{lemma}

\begin{proof}
(i)
Since the set $\argmax _{\overline \Omega} (u)$ coincides with the high ridge ({\it cf.}\  Proposition \ref{p:Yu2} (iii)) and since $\log u$ is a concave function in $\Omega$ ({\it cf.}\ Proposition \ref{p:Yu1}), we have
\[
M (\Omega)
= \{x\in\Omega:\ 0 \in D^+u(x)\}.
\]
Observe that, if $p _\e := \nabla u^\e(x) = 0$, then $p_\e \in D^+ u(x)$ 
({\it cf.} the proof of Lemma \ref{l:approx1} (i)),
hence $x \in  M(\Omega)$,
so that
\[
\argmax _{\overline \Omega _\e} u^\e \subseteq \{x \in \Omega_\e:\ \nabla u^\e(x) = 0\} \subseteq M(\Omega).
\]
On the other hand, again by using the arguments given in the proof of Lemma \ref{l:approx1} (i), one sees that, for $x \in 
\argmax _{\overline \Omega _\e} u^\e$, it holds $u ^ \e (x) = u (x)$. Since we have already proved that
$\argmax _{\overline \Omega _\e} u^\e \subseteq M (\Omega) (=\argmax _{\overline \Omega } u)$,  we deduce that $\max _{\overline \Omega _\e} u^\e = \max _{\overline \Omega} u$ (the latter equality 
can also be deduced by noticing that, for every $x \in \overline \Omega _\e$, there holds $u (x) \leq u ^ \e (x) = u ( y _x) - \frac{1}{\e} |y _x - x| ^ 2 \leq \max _{\overline \Omega} u$). 

Finally, if $x \in \argmax _{\overline \Omega } u$, the inequality $u ^ \e (x) \geq u (x) =\max _{\overline \Omega} u = \max _{\overline \Omega _\e} u^\e$ implies $x \in \argmax _{\overline \Omega _\e} u^\e$. 

Hence \eqref{f:max} is proved.

\medskip

(ii) 
From (i) it follows that 
$M_\e$ is an open neighbourhood of $M(\Omega)$,
hence we conclude that 
\begin{equation}\label{f:ming}
\min_{x\in \overline{\Omega_\e}\setminus M_\e} |\nabla u^\e(x)| =: \alpha_\e > 0.
\end{equation}
Let us define the restriction $\varphi _\e (t)  := u^\e(\Xe(t, x_\e))$, $t\geq 0$, of $u^\e$
along the gradient flow trajectory starting from $x_\e$.
Since $u^\e \in C^{1,1}$, we have that $\varphi_\e$ is continuously differentiable
with a Lipschitz continuous derivative $\dot{\varphi}_\e$. 
We have that
\[
\dot{\varphi}_\e(t) = |\nabla u ^\e(\Xe(t, x_\e))|.
\]
From \eqref{f:ming} we deduce that $\dot{\varphi}_\e\geq \alpha_\e > 0$
when $\Xe(t, x_\e) \in \overline{\Omega}_\e\setminus M_\e$,
hence the trajectory $t\mapsto \Xe(t, x_\e)$
enters $M_\e$ in a finite time $T_\e(x_\e) < \max _{\overline \Omega_\e} (u ^\e)/\alpha_\e$.

\medskip

(iii)
It follows from the fact that the vector field $F = \nabla u / |\nabla u|$
is Lipschitz continuous in $\overline{\Omega}_\e\setminus M_\e$.
\end{proof}

\begin{proposition}\label{l:approx2} Let $\Omega$ be an open bounded convex subset of $\R ^n$, 
let $u$ be a variational infinity ground state in $\Omega$, 
normalized by $\max u = \max d_{\partial\Omega}$,
and $u ^ \e$ be the supremal convolutions defined by \eqref{f:ue}.
Assume that 
\begin{equation}\label{f:stimae}
|\nabla u^\e(x_\e)|  \geq 1- b _\e  \qquad \forall x_\e \in \partial \Omega_\e\,,
\end{equation}
with $b_\e \in [0,1)$ such that $b _\e \to 0$ as $\e \to 0^+$. 

Then, for $\e>0$ small enough,  there exists $c_\e = c_\e(b_\e) \in (0,1)$ 
(with $c_\e(b) \to 0$ as $b \to 0$) 
such that, setting
\begin{equation}\label{f:Me}
M _\e := \Big \{ x \in \Omega \ :\ u ^ \e (x) > 1 - c_\e  \Big \}\, , 
\end{equation}
and denoting by $T _\e (x_\e)$ the entering time of $\Xe ( \cdot, x_\e)$ into $M _\e$ according to Lemma \ref{l:prel}  (ii), 
 for every  $x_\e \in\partial{ \Omega_\e}$ there holds: 
\begin{equation}\label{f:th}
|\nabla u^\e| (\Xe(t_1, x_\e)) \leq |\nabla u^\e|  (\Xe (t_2, x_\e) ) \qquad  \forall  \, t_1, t_2  \, \hbox{ with }  0\leq t_1 \leq t_2 \leq T _\e ( x_\e)
\,.
\end{equation}
Consequently, we have
\begin{equation}\label{f:th2}
|\nabla u^\e| \geq 1  - b _\e \qquad \text{ on } \Omega _\e \setminus M _\e.
\end{equation}
\end{proposition}

\begin{proof}
We start by proving the following claim, where $x $ is an arbitrary fixed point in $\Omega _\e$:
\begin{equation}\label{f:supine}
(p, X)\in J^{2,+} u^\e(x),\
u^\e(x) < |p| - \frac{\e |p|^2}{2}
\quad\Longrightarrow\quad
\pscal{Xp}{p}\geq 0
\quad (\text{with}\ p = \nabla u^\e(x)).
\end{equation}
Namely, let $x\in\Omega_\e$ and let
$(p, X) \in J^{2,+} u^{\e}(x)$.
By the magic property of super-jets ({\it cf.} \cite[Lemma A.5]{CHL}), taking into account that $u^\e$ is differentiable at every point,
we have that
\[
p = \nabla u^\e(x),\qquad
(p,X) \in J^{2,+} u(x+\e p).
\]

Since $u$ is a sub-solution of \eqref{ground},
by definition it holds
\[
\min\left\{\frac{|p|}{u(x+\e p)} - 1, \ -\pscal{Xp}{p}\right\}\leq 0.
\]
If $|p| > u(x+\e p)$,
that is if 
\[
u(x+\e p) = u^\e(x) + \frac{\e |p|^2}{2} < |p|\,,
\]
from the above condition we deduce that
$\pscal{Xp}{p}\geq 0$, and \eqref{f:supine} is proved.

\smallskip

Let us consider the open set
\[
Q_\e := \left\{
x\in\Omega_\e:\
u^\e(x) < |\nabla u^\e(x)| - \frac{\e}{2}\, |\nabla u^\e(x)|^2
\right\}\,.
\]
If $x\in Q_\e$ and $(p,X)\in J^{2,+} u^\e(x)$,
from the previous claim we have that $-\pscal{Xp}{p}\leq 0$,
hence $u^\e$ is a viscosity subsolution of
the equation $-\Delta_\infty u^\e \leq 0$ in $Q_\e$, 
i.e., it is $\infty$--subharmonic in $Q_\e$.

For every $x_\e\in \partial \Omega_\e$ 
let us consider the restriction $\varphi_\e (t)  := u^\e(\Xe(t, x_\e))$ of $u^\e$
along $\Xe(t, x_\e)$ for $t\geq 0$, and let
us define
\begin{equation}\label{f:taue}
\tau_\e := \sup\{ t > 0: \
\varphi _\e (s) < \dot{\varphi _\e}(s)- \frac{\e}{2}\, \dot{\varphi _\e}(s)^2
\quad\forall s\in [0,t)\}.
\end{equation}
Since 
\[
\varphi _\e(0) = m_\e  \qquad \text{ and } \quad 
\dot{\varphi _\e}(0) = |\nabla u^\e(x_\e)| \geq 1 - b _\e \,
\]
with $m_\e \to 0$ and $1- b_\e \to 1$ as $\e \to 0 ^+$, 
then, for small $\e$, every  $t>0$ sufficiently small belongs to the set at the right-hand side of \eqref{f:taue}. 
Hence $\tau_\e > 0$.

At every point $x = \Xe(t, x_\e)$ with $t\in [0, \tau_\e)$ we have 
\[
u^\e(x) < |\nabla u^\e(x)| - \frac{\e}{2}\, |\nabla u^\e(x)|^2\,, 
\]
hence $\Xe(t, x_\e) \in Q_\e$ for every $t\in [0, \tau_\e)$.

Since $u^\e$ is $\infty$--subharmonic in $Q_\e$,
from the properties of the gradient flow of
$\infty$--subharmonic functions (see e.g.\ \cite[Proposition~6.2]{Cran})
we have that
the map $t\mapsto \varphi_\e(t)$ is convex in $[0,\tau_\e]$.
In particular, $\dot{\varphi}_\e$ is a non-decreasing function in $[0,\tau_\e]$. 

Hence, 
by using the assumption $\dot{\varphi_\e }(0) \geq 1- b _\e$, and 
the
fact that the map $s \mapsto s  - \e\, s^2/2$ is 
increasing in $[0, 1/\e]$, we get
\[
\varphi _\e(\tau_\e ) = \dot{\varphi_\e }(\tau_\e )- \frac{\e}{2}\, \dot{\varphi_\e }(\tau_\e )^2
\geq 
\dot{\varphi_\e }(0)- \frac{\e}{2}\, \dot{\varphi_\e }(0)^2\geq 
1-b_\e - \frac{\e}{2} ( 1 - b_\e)^2\,.
\]
Therefore, if we consider the level set $M _\e$ defined in \eqref{f:Me} with  
\[
c _\e := b_\e + \frac{\e}{2} ( 1 - b_\e)^2\,, 
\]
at the time $\tau_\e$ the trajectory $\Xe(\cdot, x_\e)$ has
already entered $M_\e$, so that $T_\e(x_\e) < \tau_\e$.

Then the inequality \eqref{f:th} holds because $\dot{\varphi_\e}$ is a non-decreasing function in $[0,\tau_\e]$. 
\end{proof}

\bigskip

{\bf Completion of the proof of Theorem \ref{thm1}}. 

We can assume without loss of generality that $\Lambda _ \infty = 1$. 
We proceed in three steps. 

\smallskip
{\it Step 1:  It holds $c=1$}.  

Let $z \in M$, and let $y \in \Pi (z)$. The function $d _{\partial \Omega}$ is differentiable on the open ray $]y, z[$; morever, 
since $u$ is assumed to be of class $C ^ 1$ in a neighbourhood of $\partial \Omega$, 
setting $\nu := (z-y)/ |z-y|$ there exists $\d>0$ such that  $u$ is differentiable at any point of $]y, y + \d \nu[$.  
Since  $u \leq d _{\partial \Omega}$ in $\Omega$   and $u = d _{\partial \Omega}$ on $]y, z[$ ({\it cf.}\ Proposition \ref{p:Yu2}), we have that $\nabla u(x) = \nabla d(x)$ on  the segment $]y, y + \d \nu[$. Since by assumption  $u $ is of class $C ^ 1$ up to the boundary ({\it i.e.}  on  $\{ x \in \overline \Omega \ :\ 
d _{\partial \Omega} (x)  < \d \}$), we infer that
$$ c = |\nabla u (y)| = \lim _{\d \to 0} { |\nabla u ( y + \d \nu )| } = \lim _{\d \to 0}  |\nabla d_{\partial \Omega} ( y + \d \nu )|   = 1 \,.$$

\smallskip

{\it Step 2: It holds $|\nabla u |   \geq 1 \ \hbox{a.e.\ in } \Omega\,.$} 

Since $u$ is assumed to be of class $C^1$ in a neighborhood of $\partial\Omega$,
there exists $\e_0 > 0$ such that
$\partial\Omega_\e$ is contained in this neighborhood for every $\e \in (0, \e_0)$.
Moreover, since $|\nabla u| = 1$ on $\partial\Omega$, we have
\[
\min_{\partial\Omega_\e} |\nabla u^\e| =: 1 - b_\e
\qquad
\text{with}\ b_\e \to 0\ \text{as}\ \e\to 0^+.
\]
From Proposition~\ref{l:approx2} we deduce that there exists $c _\e = c _\e (b _\e)$ (with 
$c_\e (b) \to 0$ as $b \to 0$) such that, if $M _\e$ is defined by \eqref{f:Me}, there holds
\begin{equation}\label{f:estig}
|\nabla u^\e(x)| \geq 1 - b_\e \qquad
\forall x \in \Omega_\e \setminus M_\e.
\end{equation}
Moreover, $\Omega_\e \to \Omega$ and $M_\e\to M(\Omega)$ in the Hausdorff distance,
so that every point $x\in \Omega\setminus M(\Omega)$ belongs to
$\Omega_\e \setminus M_\e$ for $\e$ small enough.

Since $b_\e \to 0$, from Lemma~\ref{l:approx1} (iv) we deduce that
$|\nabla u(x)| \geq 1$ at every point $x\in\Omega\setminus M(\Omega)$ of differentiability of $u$,
hence for almost every $x\in\Omega$.

\smallskip
{\it Step 3: It holds $u = \dist$ and $\Omega$ is a stadium-like domain.} 

In view of Proposition \ref{p:Yu2} (i), to show that $u = \dist$ we have to prove only the inequality $u \geq \dist$. 
To that aim, we exploit Step 2: since $u$ satisfies the inequality $|\nabla u |   \geq 1$ a.e.\ in  $\Omega$ and is locally semiconcave ({\it cf.} Proposition~\ref{p:Yu1}), 
it turns out to be a viscosity supersolution to the eikonal equation (see \cite{CaSi}, proof of the inequality (5.16) in Proposition~5.3.1). 
Then, taking into account that $u = 0$ on $\partial \Omega$, the required inequality $u \geq \dist$ follows from the comparison principle for viscosity solutions to the eikonal equation proved in \cite[Theorem~1]{Ishii}. 

Finally, the equality $u = \dist$, combined with Theorem 2.7 in \cite{Yu} tells us that $\cut (\Omega) = \high (\Omega)$, so that $\Omega$ is a stadium-like domain. 
\qed

\section{Proof of Theorem \ref{thm2}}\label{proof2}

Under the assumptions of Theorem \ref{thm2}, in place of working with the supremal convolutions, we can consider  directly the gradient flow of the variational infinity ground state $u$, namely the family of curves solving the Cauchy problems
\begin{equation}\label{f:geo}
\begin{cases}
\dot \gamma (t) = \dfrac{\nabla u  (\gamma  (t))}{|\nabla u  (\gamma  (t))|}\,,
\\
\gamma  (0) = x  \in  \Omega \setminus M (\Omega) \,. 
\end{cases}
\end{equation}

Indeed, by the assumption $u \in C ^ {1, 1} (\Omega \setminus M(\Omega))$, for every 
$x \in\Omega \setminus M(\Omega)$  problem \eqref{f:geo} admits a unique  solution 
\[
\X(\cdot, x)\colon I^x := (T^-(x), T^+(x)) \to \Omega\setminus M(\Omega),
\]  
which  will be called a {\it trajectory} associated with $u$.

\begin{lemma}\label{l:prel2}
Let $u$ be a variational infinity ground state, and assume that $u \in C ^ {1, 1} (\Omega \setminus M(\Omega))$. 
Then
every trajectory $\X (\cdot, x)$ enters $M(\Omega) $ at the finite time $T^+(x)$,
that is,
\[
\X (t, x  ) \in \Omega \setminus M (\Omega)\quad \forall t \in [0, T^+ (x)), \qquad  
\lim_{t\nearrow T^+(x)} \X (t, x  ) \in  M (\Omega).
\]
\end{lemma}

\begin{proof}
Let \(\X(\cdot, x)\) be a trajectory,
and let us consider the function $\varphi(t) := u (\X(t,x))$.  
Taking into account that $-\Delta_\infty u = 0$ in $\Omega\setminus M(\Omega)$
(see \cite[Thm.~3.1]{Yu}), 
from the properties of the gradient flow of
$\infty$--harmonic functions (see e.g.\ \cite[Proposition~6.2]{Cran})
we conclude that $\varphi$ is affine on $[0, T^+(x))$ and, being increasing,
	the trajectory $t\mapsto \X(t,x)$ enters $M(\Omega)$ in finite time,
	{\it i.e.}\ $T^+(x) < +\infty$.
\end{proof}

\bigskip

{\bf Completion of the proof of Theorem \ref{thm2}}

We can assume without loss of generality that $\Lambda _ \infty = 1$. 
We proceed in three steps. 

\smallskip
{\it Step 1:  For every $z \in M(\Omega)$, it holds $S ^ - (z) = 1$}.  

Let $z \in M (\Omega)$ and $y \in \Pi (z)$ be fixed. The function $d_{\partial \Omega}$ is differentiable at any point of the ray $]y,z[$, and by assumption the same holds true for $u$, so that in particular we have $S ^ - (x) = |\nabla u (x)|$ for every $x \in ]y, z[$. 
Since  $u \leq d _{\partial \Omega}$ in $\Omega$   and $u = d _{\partial \Omega}$ on $]y, z[$ ({\it cf.} Proposition \ref{p:Yu2}), we have $\nabla u(x) = \nabla d _{\partial \Omega} (x)$ on  the segment $]y, z[$. Thus, 
	$$S^-(x) = |\nabla u (x)| = |\nabla d _{\partial \Omega} (x)| = 1 \qquad \forall \, x\in ]y,z[\, . $$
	Recalling that $S^-$ is continuous is $\Omega$ ({\it cf.} Proposition \ref{p:Yu3}),  
	we conclude that $S^-(z) = 1$.

\smallskip
{\it Step 2: It holds $|\nabla u |   = 1 \ \hbox{a.e.\ in } \Omega\,.$} 
	
Let us consider the gradient flow associated with $u$. 
We know from Lemma \ref{l:prel2} that, for every $x\in \Omega\setminus M (\Omega)$,  
the trajectory \(\X(\cdot, x)\) enters $M (\Omega)$ in finite time. 
By arguing as in the proof of Lemma \ref{l:prel2}, 
we see that $|\nabla u|$ is constant along it.
	By Step 1 and the continuity of $S^-$, we deduce that
	$|\nabla u(x)| = 1$
	for a.e.\ $x\in\Omega\setminus M (\Omega)$.

\smallskip
{\it Step 3: It holds $u = \dist$ and $\Omega$ is a stadium-like domain.} 

Since by Step 2 there holds  $|\nabla u |   =1$ a.e.\ in  $\Omega$, and since $u$ is locally semiconcave ({\it cf.} Proposition \ref{p:Yu1}), 
it turns out to be a viscosity solution to the eikonal equation (see \cite[Proposition 5.3.1]{CaSi}). 
Then, we can conclude the proof as done for Theorem \ref{thm1}. Namely,  since  $u = 0$ on $\partial \Omega$, the  equality $u =\dist$ follows from the comparison principle for viscosity solutions to the eikonal equation proved in \cite[Theorem 1]{Ishii}. 
Once proved that $u = \dist$, Theorem 2.7 in \cite{Yu} tells us that $\cut (\Omega) = \high (\Omega)$, so that  $\Omega$ is a stadium-like domain.

\bigskip
\noindent
\textbf{Acknowledgements.}
The authors would like to thank 
an anonymous referee for having suggested
a significant simplification of some technical proofs.

\def\cprime{$'$}

\end{document}